\documentclass[10pt,oneside]{amsart}
\usepackage{amsmath}
\usepackage{amsthm}
\usepackage{amsfonts}
\usepackage{amssymb,amscd,epsf,verbatim}
\usepackage{mathrsfs}
\usepackage{graphicx}
\usepackage{latexsym}
\usepackage{standalone}
\usepackage{lscape}
\usepackage{hyperref}
\usepackage{tikz}
\usetikzlibrary{calc}
\usetikzlibrary{matrix,arrows,decorations.pathmorphing}
\usepackage{tikz-cd}
\usepackage{color}
\usepackage{geometry}
\usepackage{stmaryrd}
\usepackage{multirow}
\usepackage{enumitem}
\usepackage{framed}
\usepackage{cite}

\newtheorem{theorem}{Theorem}[section]
\newtheorem{lemma}[theorem]{Lemma}

\newtheorem{proposition}[theorem]{Proposition}
\newtheorem{corollary}[theorem]{Corollary}

\newtheorem{mainthm}{Theorem} 
\newtheorem{lemma*}[mainthm]{Lemma}
\newtheorem{proposition*}[mainthm]{Proposition}

\theoremstyle{definition}
\newtheorem{definition}[theorem]{Definition}
\newtheorem{notation}[theorem]{Notation}

\newtheorem{question}[theorem]{Question}
\newtheorem{remark}[theorem]{Remark}
\newtheorem{example}[theorem]{Example}
\newtheorem{remark*}[mainthm]{Remark}
\newtheorem{definition*}[mainthm]{Definition}

\newcommand{\Spec}{\operatorname{Spec}}

\newcommand{\Spa}{\operatorname{Spa}}

\newcommand{\Hom}{\operatorname{Hom}}

\newcommand{\id}{{\operatorname{id}}}

\newcommand{\an}{\mathrm{an}}
\newcommand{\ad}{\mathrm{ad}}

\newcommand{\perf}{{\operatorname{perf}}}

\newcommand{\hotimes}{\hat{\otimes}}
\newcommand{\et}{\operatorname{\acute{e}t}}
\newcommand{\fet}{\operatorname{f\acute{e}t}}
\newcommand{\proet}{\operatorname{pro\acute{e}t}}
\newcommand{\profet}{\operatorname{prof\acute{e}t}}
\renewcommand{\O}{\mathcal{O}}

\newcommand{\N}{\mathbb{N}}
\newcommand{\Z}{\mathbb{Z}}
\newcommand{\Q}{\mathbb{Q}}

\newcommand{\C}{\mathbb{C}}
\newcommand{\F}{\mathbb{F}}

\title[perfectoid covers of abelian varieties]{perfectoid  covers of abelian varieties} 
\author{
	Clifford Blakestad \and
	Dami\'an Gvirtz \and
	Ben Heuer \and 
	Daria Shchedrina \and
	Koji Shimizu \and 
	Peter Wear \and
	Zijian Yao}

\begin{document}
	
	\maketitle
	
	\begin{abstract}
For an abelian variety $A$ over an algebraically closed non-archimedean field of residue characteristic $p$, we show that there exists a perfectoid space which is the tilde-limit of $\varprojlim_{[p]}A$. Our proof also works for the larger class of abeloid varieties.
	\end{abstract}

	 %%%%%%%%%%%%
%%      Section 1
%%%%%%%%%%%%
	\section{Introduction} 

Let $p$ be a prime and let $K$ be an algebraically closed non-archimedean field of residue characteristic $p$.
For an abelian variety $A$ over $K$ we consider the inverse system of $A$ under the $p$-multiplication morphism:
\[\cdots\xrightarrow{[p]}A\xrightarrow{[p]}A\xrightarrow{[p]}A.\]
Via the adic analytification functor, we may see this as an inverse system of analytic adic spaces over $\operatorname{Spa}(K,\mathcal O_K)$, where $\mathcal O_K$ is the ring of integers of $K$.
The primary goal of this article is to show that the ``inverse limit'' of this tower exists in some way and is a perfectoid space: Since inverse limits rarely exist in the category of adic spaces, in \cite{huber2013etale} Huber introduced the weaker notion of tilde-limits to remedy this problem. This is the notion of ``limits" we are going to use. More precisely, we prove the following slightly more general result:

 %Before we state the precise version of the theorem, we remark that the main assertion already implicitly appeared in [?] [?] without justification. Nevertheless we decide to fill in this gap in the literature, since the proof, despite being straight-forward, is quite subtle to spell out. 

\begin{mainthm} \label{thm:main_thm_intro}
	Let $A$ be an abeloid variety over $K$, for instance an abelian variety seen as a rigid space. Then there is a unique perfectoid space $A_\infty$ over $K$ such that
	$A_\infty \sim \varprojlim_{[p]} A$ is a tilde-limit.
\end{mainthm}

The possibility of results in this direction is mentioned in \S 7 and \S 13 of \cite{scholzeICMproceedings}, and in the case of abelian varieties with good reduction, this theorem was proven already in \cite[Lemme~A.16]{Pilloni-Stroh}. We recall the argument in Lemma~\ref{tilde-limit exists and is perfectoid in the good reduction case} below. 

In general, $A$ has semi-stable reduction by the assumption that $K$ is algebraically closed.
Consequently, the theory of Raynaud extensions provides us with a short exact sequence 
\[ 0 \rightarrow T \rightarrow E  \rightarrow  B  \rightarrow  0\]
of rigid groups, where $T = (\mathbb G_m^{\text{an}})^{d}$ is a split rigid torus and $B$ is the analytification of an abelian variety with good reduction, such that $A = E/M$ for a discrete lattice $M \subset E$. This short exact sequence is split locally on $B$, allowing us to locally write $E$ as a product of $T$ and an open subspace of $B$.
Our strategy for the proof of Theorem \ref{thm:main_thm_intro}, which more generally applies to any abeloid variety over $K$, is now similar to the good reduction case:
\begin{enumerate}
\item Construct a perfectoid tilde-limit $T_\infty\sim\varprojlim_{[p]} T$. This is easy.
\item Use $T_\infty$ and $B_\infty$ to construct a perfectoid tilde-limit $E_\infty\sim\varprojlim_{[p]} E$.
\item Study the quotient map $E\rightarrow A$ in the limit over $[p]$ to construct the desired space $A_\infty$.
\end{enumerate}

More precisely, this article is organised as follows: In \S2 we recall the definition of tilde-limits and collect some useful lemmas about tilde-limits and perfectoid spaces. In particular, we construct the perfectoid tilde-limit $T_\infty$. In \S3 we use the language of fibre bundles to construct a perfectoid tilde-limit $E_\infty$: The Raynaud extension of $A$ mentioned earlier arises from a short exact sequence of formal group schemes over $\O_K$
\[0\rightarrow \overline{T}\rightarrow \overline{E}\rightarrow \overline{B}\rightarrow 0\]
by taking generic fibres and forming the pushout with respect to the open immersion $\overline{T}_\eta\rightarrow T$. Since the sequence is locally split, we can see $\overline{E}\rightarrow \overline{B}$ as a principal $\overline{T}$-bundle and formation of $E$ amounts to a change of fibre from $\overline{T}_\eta$ to $T$. We get the desired tilde-limit by tracing the local splitting through the tower of multiplication by $[p]$. This will also show that there is a short exact sequence of perfectoid groups
\[ 0\to T_\infty \to E_\infty \to B_\infty \to 0.\]

In \S4 we finish the proof of Theorem~\ref{thm:main_thm_intro} by constructing $A_\infty$ from $E_\infty$ as follows: After choosing lattices $M\subset M_n\subset E$ that map isomorphically to $M$ under $[p^n]\colon E\rightarrow E$, the $[p]$-multiplication tower of $A=E/M$ naturally factors into two separate towers: One is the tower of maps $E/M_{n+1}\rightarrow E/M_n$ induced from $[p]$-multiplication of $E$, and the other is induced from the projection maps $v^n\colon E/M\rightarrow E/M_n$. Using local splittings, one can construct a perfectoid tilde-limit $A'_\infty\sim \varprojlim_n E/M_n$ of the first tower from $E_\infty$. It fits into a short exact sequence
\[0\to M\to E_\infty\to A'_\infty\to 0. \]
 The existence of $A_\infty\sim \varprojlim_{[p]}A$ then follows as the quotient maps $v^n\colon E/M\rightarrow E/M_n$ are \'etale. In fact, they are locally split in the analytic topology, from which one can deduce the following analogue of Raynaud uniformisation for $A_\infty$: Write $D_n$ for the kernel of $v^n$. Then there is a profinite perfectoid tilde-limit $D_\infty\sim \varprojlim_{[p]} D_n$ and a short exact sequence of perfectoid groups
\[0\rightarrow M\rightarrow D_\infty \times E_\infty \rightarrow A_\infty\rightarrow 0,\]
which we regard as an analogue of the sequence $0\rightarrow M\rightarrow E\rightarrow A\rightarrow 0$.

We give three applications of Theorem~\ref{thm:main_thm_intro} in \S 5:
As observed by Hansen, one can deduce from Theorem~\ref{thm:main_thm_intro} the existence of certain universal covers of curves by embedding them into their Jacobian:
	\begin{corollary}[Hansen,\cite{Hansen-blog}]
	Let $C$ be a connected smooth projective curve of genus $g\geq 1$ over $K$. Fix a geometric point $x\colon \Spec(K) \rightarrow C$ and for each open subgroup $H$ of $\pi_1(C,x)$, let $C_H$ denote the finite \'etale cover of $C$ corresponding to $H$. We regard $C$ and $C_H$ as analytic adic spaces.
	\begin{enumerate}
		\item There is a perfectoid tilde-limit $\tilde{C} \sim \varprojlim_{H}C_H$ where $H$ ranges over the open subgroups of $ \pi_1(C,x)$. 
		\item The morphism $\tilde{C}\to C$ is a pro-\'etale $\pi_1(C,x)$-torsor. It is universal with this property in the sense that it represents the fibre functor sending	 pro-finite-\'etale perfectoid covers $X\to C$ to the $\pi_1(C,x)$-module $F(X)=\mathrm{Hom}_C(x,X)$.
		\item For any pro-finite-\'etale morphism $X\to C$, there is a natural isomorphism
		\[ X = \underline{F(X)}\times^{\pi_1(C,x)}\tilde{C}:=(\underline{F(X)}\times \tilde{C})/\pi_1(C,x).\]
		Here the right hand side is the categorical quotient in adic spaces for the antidiagonal action.
	\end{enumerate}
\end{corollary}

Second, we note that the analogue of this corollary also works for $C$ replaced by an abelian variety, in which case the pro-\'etale fundamental group is isomorphic to the absolute Tate module $TA:=\varprojlim_{N\in \N}A[N]$. In particular, one obtains from this two different natural ways to uniformise the diamond $A^{\diamond}$ attached to $A$: On the one hand, as a consequence of Theorem~\ref{thm:main_thm_intro}, we can write
\[ A^{\diamond}=A_\infty/T_pA.\]
On the other hand, one can deduce from Theorem~\ref{thm:main_thm_intro} that there is also a perfectoid tilde-limit $\tilde{A}\sim \varprojlim_{[N]} A$ which gives rise to a natural isomorphism
\[ A^{\diamond}=\tilde{A}/T A.\]
Here the second equation describes $A$ in terms of the universal connected pro-finite-\'etale cover $\tilde{A}\to A$, whereas the first uses the universal connected pro-finite-\'etale pro-$p$-cover.
Either may be seen as a sort of analogue of Riemann uniformisation of abelian varieties over $\C$. 

Our third application of Theorem~\ref{thm:main_thm_intro} states that in line with this analogy to the complex case, the cohomology of constructible $\F_p$-sheaves on $A_\infty$ behaves like that of a Stein space: This follows in combination with a result of Reinecke:
\begin{corollary}[Reinecke]
	Let $L$ be a constructible sheaf of $\mathbb F_p$-modules on $A_{\et}$. Then for $i>\dim A$,
	\[\textstyle\varinjlim_{n\in \mathbb N}H_{\et}^i(A,[p^n]^{\ast}L)=0.\]
\end{corollary}

 The paper ends with an appendix on fibre bundles and associated fibre bundle constructions in the context of adic spaces, which develops some language that we need in the construction of $A_\infty$.

 \begin{comment}
Now we end the introduction by describing the content of each section. 

	\begin{question} \label{question_intro}
	    \begin{enumerate} 
	    \item		Given a rigid group $G$, when is there an adic space $G_\infty$ such that $G_\infty \sim  \varprojlim_{[p]} G ?$
	    \item If it exists, and $K$ is perfectoid, when is $G_\infty$ perfectoid?
	    \end{enumerate}
	\end{question}

	But before we give proofs for examples of rigid groups $G$ for which a perfectoid tilde-limit exists, we first note that the second question certainly doesn't have an affirmative answer for all rigid group varieties:
	\begin{example}
		For the additive group $\mathbb G_a^{\operatorname{an}}$, we know that $[p]$ is an isomorphism and therefore $\varprojlim_{[p]} \mathbb G_a=\mathbb G_a$ exists (even as an actual limit in the category of adic spaces) but is certainly not perfectoid.
	\end{example}

\end{comment}

 \addtocontents{toc}{\protect\setcounter{tocdepth}{0}} %some hack to hide the acknowledgements in the toc
 \section*{Acknowledgements}
 \addtocontents{toc}{\protect\setcounter{tocdepth}{2}} % end hack
 This work started as a group project at the 2017 Arizona Winter School. We would like to thank Bhargav Bhatt for proposing the project, for his guidance and for his constant encouragement, and we would like to thank Matthew Morrow for his help during the Arizona Winter School. In addition we would like to thank the organizers of the Arizona Winter School for setting up a great environment for us to participate in this project. We would like to thank David Hansen for letting us include Corollary~\ref{c:universal-covers-of-curves} which we learnt from his blog.
 
  During this work, Cliff Blakestad was partially supported by NRF-2018R1A4A1023590.
 Dami\'an Gvirtz and Ben Heuer were supported by the Engineering and Physical Sciences Research Council [EP/L015234/1], the EPSRC Centre for Doctoral Training in Geometry and Number Theory (The London School of Geometry and Number Theory), University College London. 
 During the Arizona Winter School, Daria Shchedrina was supported by Peter Scholze and the DFG.
During the preparation of this work, Koji Shimizu was partially supported by NSF grant DMS-1638352 through membership at the Institute for Advanced Study.
 Peter Wear was supported by NSF grant DMS-1502651 and UCSD and would like to thank Kiran Kedlaya for helpful discussions.

\addtocontents{toc}{\protect\setcounter{tocdepth}{0}} %some hack to hide the acknowledgements in the toc
\section*{Notation}
\addtocontents{toc}{\protect\setcounter{tocdepth}{2}} % end hack
	Let $K$ be an algebraically closed non-archimedean field, let $\mathcal O_K$ be the ring of integers of $K$ and fix a pseudo-uniformiser $\varpi\in \mathcal O_K$ such that $p\in\varpi\mathcal O_K$. 
	
	We will use adic spaces over $\operatorname{Spa}(K,\O_K)$ in the sense of Huber, and perfectoid spaces over $\operatorname{Spa}(K,\O_K)$ in the sense of Scholze \cite{perfectoid}. We denote by $X\mapsto X^{\an}$ the analytification functor from schemes of finite type over $X$ to analytic adic spaces over $(K,\O_K)$.
	
	By a rigid space, we shall always mean an analytic adic space of topologically finite type over $\operatorname{Spa}(K,\mathcal O_K)$. 
	In particular, by an open cover of a rigid space we shall always mean a cover of the associated adic space, so that we do not need the notion of admissible covers.
	
	For a $\varpi$-adic formal scheme $\mathfrak X$ over $\operatorname{Spf}(\mathcal O_K)$, we denote by $\mathfrak X_\eta:=\mathfrak X^{\mathrm{ad}}\times_{\operatorname{Spa}(\mathcal O_K,\mathcal O_K)}\operatorname{Spa}(K,\mathcal O_K)$ its adic generic fibre. Here $\mathfrak X^{\mathrm{ad}}$ is the adification in the sense of \cite{SW}.

%%%%%%%%%%%%
%%      Section 2
%%%%%%%%%%%%	
	
\numberwithin{theorem}{section}
	\section{Tilde-limits of rigid groups} \label{section:tilde_limit}

		\subsection{Tilde-limits} 
	We begin with some lemmas on tilde-limits that we will need throughout.
		
	Inverse limits often do not exist in the category of adic spaces, and neither do they in rigid spaces. Instead we use the notion of tilde-limits from \cite[Definition 2.4.2]{huber2013etale}:	
	\begin{definition} 
Let $(X_i)_{i\in I}$ be a filtered inverse system of adic spaces with quasi-compact and quasi-separated transition maps, and let $X$ be an adic space with a compatible system of morphisms $f_i\colon X \rightarrow X_i$. We write $X \sim \varprojlim X_i$ and say that $X$ is a \textbf{tilde-limit} of the inverse system $(X_i)_{i\in I}$ if the map of underlying topological spaces $|X| \rightarrow \varprojlim |X_i|$ is a homeomorphism, and there exists an open cover of $X$ by affinoids $\operatorname{Spa} (A, A^+) \subset X$ such that the map 
$$ \varinjlim_{\operatorname{Spa}(A_i, A_i^+) \subset X_i} A_i \rightarrow A$$
has dense image, where the direct limit runs over all $i\in I$ and all open affinoid subspaces $\operatorname{Spa}(A_i, A_i^+) \subset X_i$ through which the morphism $\operatorname{Spa}(A, A^+) \subseteq X\rightarrow X_i$ factors.
	\end{definition}
	
	\begin{remark} \label{remark:tilde_limit_non_unique}
As pointed out after Proposition 2.4.4 of \cite{SW}, tilde-limits (if they exist) are in general not unique. However, Corollary~\ref{corollary: perfectoid tilde limit is unique} below says that perfectoid tilde-limits are unique.
	\end{remark}

We recall a few results from \cite{SW}, \S2.4 on tilde-limits that we will use frequently throughout:

\begin{proposition}[\cite{SW}, Proposition 2.4.2]\label{SW Proposition 2.4.2}
	Let $(A_i,A_i^{+})$ be a direct system of affinoids over $(K,\mathcal O_K)$ with compatible rings of definition $A_{i,0}$ carrying the $\varpi$-adic topology. Let $(A,A^{+})=(\varinjlim A_i,\varinjlim A_i^{+})$ be the affinoid algebra equipped with the topology making $\varinjlim A_{i,0}$ with the $\varpi$-adic topology a ring of definition. Then
	\[\operatorname{Spa}(A,A^{+})\sim \textstyle\varprojlim \operatorname{Spa}(A_i,A_i^{+}).\]
\end{proposition}
\begin{proposition}[\cite{SW}, Proposition 2.4.3]\label{SW Proposition 2.4.3}
	Let $X\sim \varprojlim_{i\in I} X_i$ be a tilde-limit and let $U_i\hookrightarrow X_i$ be an open immersion for some $i\in I$. Set $U_j:=U_i\times_{X_i}X_j$ for $j\geq i$ and $U:=U_i\times_{X_i}X$. Then 
	\[U\sim \textstyle\varprojlim_{j\geq i} U_j.\]
\end{proposition}

\begin{proposition}[\cite{SW}, Proposition 2.4.5]\label{SW Proposition 2.4.5}
	Let $(X_i)_{i\in I}$ be an inverse system of adic spaces over $(K,\mathcal O_K)$ and assume that there is a perfectoid space $X$ such that $X\sim \varprojlim_{i\in I} X_i$. Then for any perfectoid $K$-algebra $(B,B^{+})$, 
	\[X(B,B^{+})  = \textstyle\varprojlim_{i\in I}X_i(B,B^{+}).\]
\end{proposition}
\begin{corollary}\label{corollary: perfectoid tilde limit is unique}
	Any two perfectoid spaces that are tilde-limits of the same inverse system of adic spaces over $(K,\mathcal O_K)$ are canonically isomorphic.
\end{corollary}
In the situation of the corollary, we will also refer to such a perfectoid space as \textit{the} perfectoid tilde-limit of the inverse system. Of course perfectoid tilde-limits do not always exist. An example of a basic situation in which they do is the following:
\begin{lemma}\label{l:pro-finite-perfectoid-spaces}
	Let $(S_i)_{i\in I}$ be an inverse system of finite sets. Let $S=\varprojlim_{i\in I} S_i$. Then the system of constant groups $\underline{S_i}=\mathrm{Spa}(\mathrm{Map}_{\mathrm{cts}}(S_i,K),\mathrm{Map}_{\mathrm{cts}}(S_i,\O_K))$ has a perfectoid tilde-limit	\[\underline{S}:=\mathrm{Spa}(\mathrm{Map}_{\mathrm{cts}}(S,K),\mathrm{Map}_{\mathrm{cts}}(S,\O_K))\sim\textstyle\varprojlim_{i\in I} \underline{S_i}.\]
\end{lemma}
\begin{proof}
	Since $S$ is compact, $\mathrm{Map}_{\mathrm{cts}}(S,K)=\mathrm{Map}_{\mathrm{cts}}(S,\O_K)[\tfrac{1}{\varpi}]$. Perfectoidness now follows from $\mathrm{Map}_{\mathrm{cts}}(S,\O_K)/\varpi=\mathrm{Map}_{\mathrm{lc}}(S,\O_K/\varpi)$. The tilde-limit property follows from Proposition~\ref{SW Proposition 2.4.2}.
\end{proof}
We will need the following basic lemma later on.

	\begin{lemma}\label{affinoid tilde-limits commute with fibre products}
		Let $(A_i, A_i^+)$ and $(B_i, B_i^+)$ be direct systems of affinoids over $(K, \mathcal O_K)$ with compatible rings of definition $A_{i,0}$ and $B_{i,0}$ carrying the $\varpi$-adic topology. Assume that there are perfectoid tilde-limits $\operatorname{Spa}(A, A^+)\sim \varprojlim \operatorname{Spa}(A_i, A_i^+)$ and $\operatorname{Spa}(B, B^+)\sim \varprojlim \operatorname{Spa}(B_i, B_i^+)$. Then \[\operatorname{Spa}(A, A^+)\times_{\operatorname{Spa}(K, \mathcal O_K)}\operatorname{Spa}(B, B^+)\sim\varprojlim (\operatorname{Spa}(A_i, A_i^+)\times_{\operatorname{Spa}(K, \mathcal O_K)} \operatorname{Spa}(B_i, B_i^+))\]
		is also a perfectoid tilde-limit.
	\end{lemma}
	\begin{proof}
		The fibre product $\operatorname{Spa}(A, A^+)\times_{\operatorname{Spa}(K, \mathcal O_K)}\operatorname{Spa}(B, B^+)$ exists and is perfectoid by \cite[Prop 6.18]{perfectoid}. In fact, it is represented by $\operatorname{Spa}(C,C^+)$, where $C=A\widehat{\otimes}_KB$ and $C^+$ is the $\varpi$-adic completion of the integral closure of the image of $A^+\otimes_{\mathcal O_K}B^+$.
		
		We first check the condition on topological spaces:
		 Since fibre products commute with limits in the category of sheaves, it follows from Proposition~\ref{SW Proposition 2.4.5} that for any perfectoid field $(D,D^+)$ over $(K,\O_K)$, we have 
\[
 (\operatorname{Spa}(A, A^+)\times_{\operatorname{Spa}(K, \mathcal O_K)}\operatorname{Spa}(B, B^+))(D,D^+)=\varprojlim (\operatorname{Spa}(A_i, A_i^+)\times_{\operatorname{Spa}(K, \mathcal O_K)} \operatorname{Spa}(B_i, B_i^+))(D,D^+).
\]
 Since the topological space can be reconstructed from this data, it follows that the underlying topological spaces of both sides coincide.
		
		It remains to check that if $\varinjlim A_i \rightarrow A$ has dense image and $\varinjlim B_i \rightarrow B$ has dense image, then $\varinjlim (A_i\otimes B_i) \rightarrow A\otimes B$ has dense image. As direct limits commute with tensor products, we have $\varinjlim (A_i\otimes B_i) = (\varinjlim A_i)\otimes (\varinjlim B_i)$. Now density can be checked directly on elements. 
	\end{proof}

\subsection{Perfectoid tilde-limits for rigid groups}

One reason why perfectoid tilde-limits along group homomorphisms are particularly interesting is that these again have a group structure:

\begin{definition}
	A \textbf{perfectoid group} is a group object in the category of perfectoid spaces.
\end{definition}
The category of perfectoid spaces over $K$ has finite products, so this is a well-defined notion.

\begin{lemma}\label{perfectoid tilde-limit is perfectoid group in a functorial way}
	Let $(G_i)_{i\in I}$ be an inverse system of adic groups such that the transition maps are homomorphisms of adic groups. Assune that there is a perfectoid tilde-limit $G_\infty\sim \varprojlim_{i\in I}G_i$.
	\begin{enumerate}
		\item  There is a unique way to endow $G_\infty$ with the structure of a perfectoid group in such a way that all projections $G_\infty\rightarrow G$ are group homomorphisms
		\item Given a morphism of inverse systems of adic groups $(G_i)_{i\in I}\to (H_j)_{j\in J}$ and a perfectoid tilde-limit $H_\infty\sim\varprojlim_{j\in J}H_j$, there is a unique morphism of perfectoid groups $G_\infty\rightarrow H_\infty$
		commuting with all projection maps.
	\end{enumerate}
\end{lemma}
\begin{proof}
	These are all consequences of the universal property of the perfectoid tilde-limit, Proposition~\ref{SW Proposition 2.4.5}, which shows that one can argue like in the case of categorical limits.
\end{proof}

Let $G$ be an adic group locally of finite type over $(K,\O_K)$, that is, a group object in the category of adic spaces over $\Spa(K,\O_K)$. Throughout we will always consider commutative groups. The main topic of study of this work is the $[p]$-multiplication tower
\[ \cdots\xrightarrow{[p]}G\xrightarrow{[p]}G.\]
We will usually assume that $G$ is $p$-divisible, i.e.\ that $[p]\colon G\to G$ is surjective.
\begin{question}\label{qu:tilde-limits-of-adic-groups}
	When is there a perfectoid space $G_\infty$ such that $G_\infty \sim \varprojlim_{[p]} G$ is a tilde-limit?
\end{question}

We are primarily interested in the following examples:
\begin{enumerate}	 
	\item Analytifications over $\Spa(K,\O_K)$ of finite type group schemes over $K$. Examples include analytifications of abelian varieties $A$ over $K$ and of tori $T$ over $K$.
	\item Generic fibres of locally topologically finite type formal group schemes over $\mathcal O_K$.
	\item Raynaud's covering space $E$  of an abelian variety with semi-stable reduction.
\end{enumerate}
\begin{remark}
	More generally, one could ask Question~\ref{qu:tilde-limits-of-adic-groups} for families of abelian varieties over $\Spec(R)$ where $R$ is any perfectoid ring. Considering the fibers of such a family in any point of $\Spa(R,R^\circ)$ motivates to also study analytifications over $\Spa(K,K^+)$ where $K^+$ is any open bounded integrally closed subring of $\O_K$. However, one can reduce this case to the one of $K^+=\O_K$.
	
	Indeed, this follows from the following technical observation:
	Let $(X_i)_{i\in I}$ be an inverse system of adic spaces $X_i$ of finite type over $(K,K^+)$ with finite transition maps. Let $X_{i,\eta}:=X_i\times_{\Spa(K,K^+)}\Spa(K,\mathcal O_K)$. Then the following are equivalent: 
	\begin{enumerate}
	\item There is a perfectoid tilde-limit $X_{\infty}\sim \varprojlim_{i \in I}X_{i}$. 
	\item There is a perfectoid tilde-limit $X_{\infty,\eta}\sim \varprojlim_{i \in I}X_{i,\eta}$.
	\end{enumerate}
We will therefore restrict attention to the case of  $K^+=\O_K$ without loss of generality.
\end{remark}

As we have already mentioned in the introduction, Question~\ref{qu:tilde-limits-of-adic-groups} has an affirmative answer in the case of abelian varieties of good reduction by \cite[Lemme~A.16]{Pilloni-Stroh}. More generally:
\begin{lemma}\label{tilde-limit exists and is perfectoid in the good reduction case}
		Let $\mathfrak G$ be a flat commutative formal group scheme over $\O_K$ such that $[p]:\mathfrak G\to\mathfrak G$ is affine. Let $G=\mathfrak G^{\ad}_{\eta}$ be the adic generic fibre. Then $G_\infty := (\varprojlim_{[p]}\mathfrak G)^{\ad}_\eta$ is a perfectoid tilde-limit
		\[G_\infty\sim \varprojlim_{[p]} G. \]
		In particular, if $B$ is an abelian variety of good reduction over $K$, there is a perfectoid tilde-limit $B_\infty\sim \varprojlim B$.
	\end{lemma}
\begin{proof}
	This holds by the same proof as in \cite[Lemme~A.16]{Pilloni-Stroh}, (see also Exercise 4 -- 6 in \cite{Bhatt}):
	 Let $\varpi\in\mathcal O_K$ be a pseudo-uniformiser such that $p\in \varpi\mathcal O_K$. The assumption that $[p]:\mathfrak G\to \mathfrak G$ is affine ensures that the limit $\mathfrak G_{\infty}:=\varprojlim_{[p]}\mathfrak G$ exists.
	 
	The mod $\varpi$ special fibre $\tilde{G} = \mathfrak G\times \Spec(\O_K/\varpi)$ is a group scheme over $\mathcal O_K/\varpi$, so the map $[p]\colon\tilde{G}\rightarrow \tilde{G}$ factors through the relative Frobenius map. Consequently, the mod $\varpi$ special fibre $\varprojlim_{[p]}\tilde{G}$ of $\mathfrak G_{\infty}$ is relatively perfect over $\mathcal O_K/\varpi$. This implies that the adic generic fibre of $\mathfrak G_{\infty}$ is perfectoid by \cite[Theorem~5.2]{perfectoid}.
\end{proof}
\begin{lemma}
	Let $T$ be a torus over $K$. Then there is a perfectoid tilde-limit $T_\infty\sim \varprojlim_{[p]} T$.
\end{lemma}
\begin{proof}
Since we assume $K$ algebraically closed, we may choose a splitting $T\cong (\mathbb{G}_m^{\an})^d$ for some $d\in \N$. By Lemma~\ref{affinoid tilde-limits commute with fibre products}, it suffices to consider the case of $d=1$. For this, we may use the open embedding $\mathbb G_m^{\an}= \mathbb P^{1,\an}\setminus\{0, \infty\}\subseteq \mathbb P^{1,\an}$. Sending $(x:y)\mapsto (x^p:y^p)$ defines a morphism $\varphi:\mathbb P^{1,\an}\to \mathbb P^{1,\an}$. The pullback of $\varphi$ to $\mathbb G_m^{\an}$ is precisely $[p]:\mathbb G_m^{\an}\to \mathbb G_m^{\an}$. We can therefore apply Proposition \ref{SW Proposition 2.4.3} to the perfectoid tilde-limit $\mathbb P_1^{\text{perf}}\sim \varprojlim_{\varphi} \mathbb P^{1,\an}$ introduced in \cite{perfectoid}. 
\end{proof}

%\begin{remark} What we aim to prove in the rest of this write-up is that for a Raynaud extension $0\rightarrow T\rightarrow E\rightarrow B\rightarrow 0$, there is a $[p]$-$F$-model for $T$ which induces a $[p]$-$F$-model for $E$. This will prove that tilde-limits $T_\infty$ and $E_\infty$ exist and are perfectoid if $K$ is perfectoid.  
%\end{remark}
			
	\begin{example}
		Regarding Question~\ref{qu:tilde-limits-of-adic-groups}, we note
		that if $G$ is not $p$-divisible, $\varprojlim_{[p]}G$ might have a tilde-limit for trivial reasons: For example, let $\mathfrak G_a$  be the $p$-adic completion of the affine group scheme $\mathbb G_a$ over $\mathcal O_K$. Then the trivial group $\operatorname{Spa}(K,\mathcal O_K)\sim \varprojlim_{[p]}(\mathfrak G_a)^{\ad}_{\eta}$ is a perfectoid tilde-limit.
	\end{example}

 %%%%%%%%%%%%%%%%%
%%%%%%%%%%%%%%%%%
%%%  Section 3
%%%%%%%%%%%%%%%%%
%%%%%%%%%%%%%%%%%

	\section{Perfectoid tilde-limits of Raynaud extensions}\label{Raynaud extensions as principal bundles of formal and rigid spaces}
	In this section we study the $p$-multiplication tower of the Raynaud extensions associated to abeloid varieties over an algebraically closed perfectoid field $K$. The main result of this section is Proposition \ref{p-F-formal tower exists for E}, which shows that the existence of perfectoid tilde-limits is closed under analytic-locally split extensions, and thus there exists a perfectoid tilde-limit $E_\infty \sim \varprojlim_{[p]} E$.

	\begin{remark}\label{Remark on dealing with general perfectoid fields by Galois descent}
		Everything in this section also works with minor modifications over a general perfectoid field. But we opt to work over an algebraically closed field to simplify the exposition.
	\end{remark}

	\subsection{Raynaud extensions}
	
        We briefly sketch the theory of Raynaud extensions here, and refer the readers to \cite{Bosch-Lut, Lut-survey, Lut} for more details on the setup.

	Let $A$ be an abelian variety over $K$. There exists a unique open rigid analytic subgroup $\overline A$ of $A$ such that $\overline A$ admits a formal model $\overline E$ that is a connected smooth $\mathcal O_K$-group scheme fitting into a short exact sequence of formal group schemes
	\begin{equation}\label{formal Raynaud extension}
	0\rightarrow \overline T \rightarrow \overline E \xrightarrow{\pi} \overline{B}\rightarrow 0,
	\end{equation}
	where $\overline{B}$ is a formal abelian scheme over $\mathcal O_K$ with rigid generic fibre $B:=\overline{B}_\eta$, and $\overline{T}$ is the completion of a torus $T_{\mathcal O_K}$ of rank $r$ over $\mathcal O_K$.
	We set $T:=T_{\mathcal O_K}\otimes_{\mathcal O_K}K$ and denote its analytification also by $T$. Then the rigid generic fibre $\overline{T}_\eta$ of the formal torus $\overline{T}$ canonically embeds into $T$. This induces a pushout exact sequence in the category of rigid groups: More precisely, there exists a rigid group variety $E$ such that the following diagram commutes and the left square is a pushout:
		\begin{equation}\label{Raynaud diagram}
		\begin{tikzcd}
			0 \arrow[r] & \overline{T}_\eta \arrow[d, hook] \arrow[r] & \overline{E}_\eta \arrow[d, hook] \arrow[r] & \overline{B}_\eta \arrow[d,equal] \arrow[r] & 0 \\
			0 \arrow[r] & T \arrow[r] & E \arrow[r] & B \arrow[r] & 0.
		\end{tikzcd}
		\end{equation}
	
	The abelian variety $A$ can be uniformized in terms of $E$ as follows:
	
	\begin{definition}
		A subset $M$ of a rigid space $G$ is called \textbf{discrete} if the intersection of $M$ with any affinoid open subset of $G$ is a finite set of points.
		Let $G$ be a rigid group, then a \textbf{lattice} in $G$ of rank $r$ is a discrete subgroup $M$ of $G$ which is isomorphic to the constant rigid group $\underline{\mathbb Z^r}$. 
	\end{definition}
	
	\begin{theorem}\label{Raynaud uniformisation}
		There exists a lattice $M \subset E$ of rank equal to the rank $r$ of the torus such that the quotient $E/M$ exists as a rigid space and such that there is a natural isomorphism
		\[A=E/M\]
making $E\rightarrow E/M=A$ a rigid group homomorphism. 
	\end{theorem}
	
	The data of the extension~(\ref{formal Raynaud extension}) together with the lattice $M\subset E$ is what we refer to as a Raynaud uniformisation of $A$. This will be the only input we need to construct the perfectoid tilde-limit $A_\infty$. Consequently, our method generalises to the class of rigid groups which admit Raynaud uniformisation, namely to abeloid varieties:
	\begin{theorem}[L\"utkebohmert, \cite{Lut}, Theorem 7.6.4]\label{Raynaud uniformisation for abeloids}
		Let $A$ be an abeloid variety, that is, a connected smooth proper commutative rigid group over $K$. Then $A$ admits a Raynaud uniformisation.
	\end{theorem}
	
	In the situation of Raynaud uniformisation, since $M$ is discrete, the local geometry of $A$ is determined by the local geometry of $E$. We will therefore first study the $[p]$-multiplication tower of $E$ in the rest of this section and will then deduce properties of the $[p]$-multiplication tower of $A$ in the next section.

	 Our strategy is to study the local geometry of $E$ and $\overline{E}$ via $T$ and $B$. An obstacle in doing this is that the categories of formal or rigid groups are not abelian, which makes working with short exact sequences a subtle issue. Another issue is that we would like to work locally on $B$, but the notion of short exact sequences does not make sense if we replace $B$ by an open $U\subseteq B$ which might not itself have a group structure.
	Instead, we have the following crucial lemma, which says that one may regard Raynaud extensions as $T$-torsors of formal schemes.

	\begin{lemma}\label{formal Raynaud sequence is locally split}
		The short exact sequence (\ref{formal Raynaud extension}) admits local sections, that is there is a cover of $\overline{B}$ by formal open subschemes $\overline{U}_i$ such that there exist local sections $s:\overline{U}_i\rightarrow \overline{E}$ of $\pi$. In particular, one can cover $\overline{E}$ by formal open subschemes of the form $\overline{T}\times \overline{U}_i\hookrightarrow \overline{E}$.
	\end{lemma}
	\begin{proof}
		This is proved in Proposition A.2.5 in~\cite{Lut}, where it is fomulated in terms of the group $\operatorname{Ext}(B,T)$. Also see \cite{BL}, \S 1.
	\end{proof}

	\begin{remark}
	In the following, we will freely work with fibre bundles of formal schemes and rigid and perfectoid spaces. For some background material on these we refer to Appendix~\ref{s:appendix}.
	\end{remark}
	
	The sequence~(\ref{formal Raynaud extension}) gives rise to a principal $\overline{T}$-bundle
	$\overline{E}\rightarrow \overline{B}$. The fact that $E$ is obtained from $\overline{E}_\eta$ via push-out along $\overline{T}_\eta\rightarrow T$ can be expressed in terms of the associated fibre bundle by saying that $E = T\times^{\overline{T}_\eta}\overline{E}_\eta$ in the sense of Definition~\ref{definition of Borel construction}.
	\begin{definition}
		We call a sequence of adic groups $0\to T\to E\xrightarrow{\pi} B\to 0$ an analytic-locally split short exact sequence if $T$ is the kernel of $\pi$ and $\pi:E\to B$ is a principal $T$-torsor in the analytic topology. This implies that $B$ is the categorial quotient in adic groups of $T\to E$.
	\end{definition}
	In particular, any Raynaud extension is an analytic-locally split short exact sequence. We can use this to prove the main result of this section, namely that $E_\infty$ is perfectoid. More generally:
	\begin{proposition}[Existence of $G_\infty$ closed under extensions]\label{p-F-formal tower exists for E}
		
		Let $0\to T\to E\to B\to 0$ be any analytic-locally split short exact sequence of adic groups over $K$. Assume that there are perfectoid tilde-limits $T_\infty \sim \varprojlim_{[p]} T$ and $B_\infty\sim \varprojlim_{[p]}B$. Then
		there is a perfectoid tilde-limit $E_\infty\sim \varprojlim_{[p]}E$. It fits into an analytic-locally split short exact sequence of perfectoid groups 
		\[0\to T_\infty\to E_\infty\to B_\infty\to 0.\]
	\end{proposition}
	\begin{proof}
	We can split $[p^n]:E\to E$ into two morphisms as follows:
			\begin{center}
				\begin{tikzcd}
				0 \arrow[r] & T \arrow[r] \arrow[d, "{[p^n]}"] & E \arrow[r, "\pi"] \arrow[d, "{[p^n]\times \pi}"] & B \arrow[d,equal] \arrow[r] & 0 \\
				0 \arrow[r] & T \arrow[r] \arrow[d,equal] & {E\times_{B,[p^n]} B} \arrow[r] \arrow[d] & B \arrow[d, "{[p^n]}"] \arrow[r] & 0 \\
				0 \arrow[r] & T \arrow[r] & E \arrow[r] & B \arrow[r] & 0
				\end{tikzcd}
				\end{center}
			Here the bottom morphism of short exact sequences is the base-change of the $T$-bundle $E\to B$ along $[p^n]:B\to B$, whereas the top morphism of short exact sequences is the change of fibre along $[p^n]:T\to T$. It is clear that the vertical maps in the middle compose to $[p^n]:E\to E$.
			
			Let now $U\subseteq B$ be an open subspace on which $E\to B$ is split, then $E|_U=T\times U$. Let $U_n$ be the pullback of $U$ along $[p^n]$. Then it follows from the universal property of the fibre product that $E\times_{B,[p^n]}B|_{U_n}\to U_n$ is again split. Consequently, the pullback of the entire diagram to $U$ becomes
					\begin{center}
							\begin{tikzcd}
							0 \arrow[r] & T \arrow[r] \arrow[d, "{[p]}"] & T\times U_n \arrow[r, "\pi"] \arrow[d, "{[p]\times \id}"] & U_n \arrow[d,equal] \arrow[r] & 0 \\
							0 \arrow[r] & T \arrow[r] \arrow[d,equal] & T\times U_n \arrow[r] \arrow[d, "{\id \times [p]}"] & U_n \arrow[d, "{[p]}"] \arrow[r] & 0 \\
							0 \arrow[r] & T \arrow[r] & T\times U \arrow[r] & U \arrow[r] & 0.
							\end{tikzcd}
						\end{center}
				It follows that the inverse system \[\cdots \xrightarrow{[p]}E\xrightarrow{[p]}E\]
			 	restricts over the open $E|_U\subseteq E$ on the base to the inverse system
				\[\xrightarrow{[p]\times [p]}T\times U_n\xrightarrow{[p]\times [p]}T\times U_{n-1}\xrightarrow{[p]\times [p]}\cdots \xrightarrow{[p]\times [p]}E_{|U}. \]
				To prove the theorem, it suffices to show that this system has a perfectoid tilde-limit.
				As there is a perfectoid tilde-limit $B_\infty\sim \varprojlim_{[p]} B$, by pullback to $U$ there is also a perfectoid tilde-limit $U_\infty\sim \varprojlim U_n$. Since there is also a perfectoid tilde-limit $T_\infty\sim \varprojlim_{[p]} T$,  Lemma~\ref{affinoid tilde-limits commute with fibre products} says that the above system has perfectoid tilde-limit $T_\infty\times U_\infty$, as desired.
				
				By Proposition~\ref{SW Proposition 2.4.5}, the maps $0\to T\to E\to B\to 0$ in the limit over $[p]$ now produces maps $0\to T_\infty \to E_\infty\to B_\infty\to 0$. By the above local description of $E_\infty$, the latter is an analytic-locally split short exact sequence.
	\end{proof}
	
	\begin{remark}
		In the case of Raynaud extensions, there is an alternative proof that also constructs a formal model $\mathfrak E_\infty$ of $E_\infty$, like in Lemma~\ref{tilde-limit exists and is perfectoid in the good reduction case}. For this, one first takes a sequence of formal models 
		\[\cdots \to \mathfrak T_{2}\xrightarrow{[\mathfrak p]_1}\mathfrak T_1\xrightarrow{[\mathfrak p]_1} \mathfrak T_0\]
		of $\cdots\xrightarrow{[p]} T\xrightarrow{[p]} T$. This can be done in such a way that each $[\mathfrak p]_i$ reduces to the relative Frobenius mod $p$. Then $\mathfrak T_\infty:=\varprojlim_{[\mathfrak p]_i}\mathfrak T_i$ is a formal model of the perfectoid space $T_\infty$ (giving an alternative proof that $T_\infty$ is a perfectoid tilde-limit). When we set $\mathfrak E_i:=\mathfrak T_i\times^{\overline{T}}\overline{E}$, we get an inverse system
		\[\cdots \to \mathfrak E_{2}\xrightarrow{[\mathfrak p]_1}\mathfrak E_1\xrightarrow{[\mathfrak p]_1} \mathfrak E_0\]
		with transition maps that factor through the relative Frobenius map mod $p$. Thus the generic fibre of $\mathfrak E_\infty:=\varprojlim_{[\mathfrak p]_i}\mathfrak E_i$ is a perfectoid tilde-limit of $\cdots \xrightarrow{[p]}E\xrightarrow{[p]}E$.
	\end{remark}
	
	\begin{remark}\label{general fields for E}
	With some work, the arguments in this section can be extended to any perfectoid base field. For instance, the Raynaud uniformisation of Theorem \ref{Raynaud uniformisation} might only be defined over a finite extension $L$ of $K$. Our argument then gives a perfectoid space over the (necessarily perfectoid) field $L$. We can then use Galois descent to get a perfectoid space over our original field $K$. This uses that the quotient of a perfectoid space by a finite group often remains perfectoid: see Theorem 1.4 of \cite{Hansen_quotients} for details. Finally, one checks that this Galois descent commutes with tilde-limits. 
	\end{remark}

%%%%%%%%%%%%%%%%%
%%%%%%%%%%%%%%%%%
%%%  Section 4
%%%%%%%%%%%%%%%%%
%%%%%%%%%%%%%%%%%
	
	\section{The case of abeloid varieties}\label{The case of abeloid varieties}
	We now prove Theorem~\ref{thm:main_thm_intro}, building on the preceding sections. Recall our setup: Let $A$ be an abeloid variety over $K$. Let $E$ be the Raynaud extension associated to $A$ from Proposition~\ref{Raynaud uniformisation for abeloids}, which is an extension of an abeloid variety $B$ of good reduction by a split rigid torus $T$ of rank $r$, and $M\subset E$  is a lattice of rank $r$ such that $A=E/M$. 

By Proposition~\ref{Raynaud uniformisation for abeloids}, the quotient map $\pi\colon E\to A$ is locally split in the analytic topology on $A$: As the action of $M$ on $E$ is totally discontinuous, for every point  $x\in A$ there is an open neighbourhood $U'$ of $E$ such that $\pi$ maps isomorphically onto an open $U:=\pi(U')$ containing $x$. Here we are careful to distinguish $U'\subset E$ and $U\subset A$, even though the two are isomorphic via $\pi$.

We fix from now on a cover $\mathfrak U$ of $A$ by opens $U$ of this form.

The pullback of $U'$ along $[p]\colon A\to A$ will in general be bigger than the pullback of $U$ along $[p]:E\to E$: e.g. in characteristic 0, the first is an \'etale $A[p]$-torsor, whereas  the latter is an \'etale $E[p]$-torsor, and by the Snake Lemma we have a short exact sequence
\[0\to E[p]\to A[p]\to M/pM\to 0\]

To relate the pullbacks, we subdivide the tower 
\[
\cdots\xrightarrow{[p]}A \xrightarrow{[p]}A\xrightarrow{[p]}A
\]
into two partial towers. For this we make some auxiliary choices: Since $K$ is algebraically closed, we can choose lattices $M_n\subseteq E$ such that $M_0=M$ and $[p]\colon E\rightarrow E$ restricts to isomorphisms $M_{n+1}\rightarrow M_n$ for all $n$.
	
	\begin{remark}\label{remark: Definition of the D_n}
		Such a choice is equivalent to the choice of subgroups $D_n\subseteq A[p^n]$ of order $p^{rn}$ for all $n$ such that $pD_{n+1}=D_n$ and $D_n+E[p^n]=A[p^n]$. Namely,
		given the lattices $M_{n}$, we obtain the desired torsion subgroups by setting $D_n:=M_{n}/M$. This is because any such lattice gives a splitting of the short exact sequence $0\rightarrow E[p^n]\rightarrow A[p^n]\rightarrow M/p^nM \rightarrow 0$.
		
		Conversely, given subgroups $D_n\subseteq A[p^n]$ with properties as above, we recover $M_n$ as the kernel of $E\rightarrow A\rightarrow A/D_n$.
		
		One might call the choice of $D_n$ for all $n$ a partial anticanonical $\Gamma_0(p^\infty)$-structure, because if $B$ admits a canonical subgroup (that is, if it satisfies a condition on its Hasse invariant), the choice of a (full) anticanonical $\Gamma_0(p^\infty)$-structure on $A$ is equivalent to the choice of a partial anticanonical $\Gamma_0(p^\infty)$-structure on $A$ and an anticanonical $\Gamma_0(p^\infty)$-structure on $B$. Note however that $A$ always has a partial anticanonical subgroup even if $B$ does not have a canonical subgroup.
	\end{remark}
	
	Following the remark, denote by $D_n$ the torsion subgroup $M_n/M\subset A$. The quotient $A_n:=A/D_n = E/M_n$ is then another abeloid variety over $K$ and the quotient map $v^n\colon A=E/M\rightarrow A_n=E/M_n$ is an isogeny of degree $p^{rn}$  through which  $[p^n]\colon A\rightarrow A$ factors. The $[p]$-multiplication tower now splits into two towers, one written vertically, the other horizontally:
		\begin{equation}\label{p-multiplication tower of E/M splits into vertical and horizontal tower}
		\begin{tikzcd}[column sep={1.1cm,between origins},row sep={0.6cm,between origins}]
			\ddots \arrow[rd] &  &  & \vdots &  & \vdots \\
			& A \arrow[rr,"v"] \arrow[rrdd, "{[p]}"'] &  & A_1 \arrow[rr,"v"] \arrow[dd,"{[p]_E}"] &  & A_2 \arrow[dd,"{[p]_E}"] \\
			&  &  &  &  &  \\
			&  &  & A \arrow[rrdd, "{[p]}"'] \arrow[rr,"v"] &  & A_1 \arrow[dd,"{[p]_E}"] \\
			&  &  &  &  &  \\
			&  &  &  &  & A.
		\end{tikzcd}
		\end{equation}
		Since each $D_n=M_n/M$ is finite \'etale, all horizontal maps are finite \'etale. The vertical tower on the other hand fits into a commutative diagram which compares it to the $[p]$-tower of $E$:
		\begin{equation}\label{F-tower for E/M}
		\begin{tikzcd}[column sep={1.1cm,between origins},row sep={0.5cm,between origins}]
			&\vdots&\vdots&\vdots&\\
			0 \arrow[r] & M_1 \arrow[dd, "\cong"] \arrow[r] & E \arrow[dd, "{[p]}"] \arrow[r] & A_1 \arrow[dd, "{[p]_E}"] \arrow[r] & 0 \\
			\\
			0 \arrow[r] & M \arrow[r] & E \arrow[r] & A \arrow[r] & 0.
		\end{tikzcd}
		\end{equation}
		\begin{definition}
			Let $M_\infty:=\varprojlim_{n\in\N} M_n$ be the limit of the left vertical tower. 
		\end{definition}
		We note that $M_\infty$ is an actual limit, not just a tilde-limit, because the transition maps are isomorphisms. In particular, the projection $M_\infty\to M$ is an isomorphism as well. 
		By Proposition \ref{SW Proposition 2.4.5}, we get a natural map $M_\infty\to E_\infty$. 
		\begin{proposition}\label{M_infty->E_infty->A'_infty}
			There is a perfectoid tilde-limit $A'_\infty\sim \varprojlim_{n\in\N} A_n$. It fits into an analytic-locally split short exact sequence of perfectoid groups 
			\[0\to M_\infty\to E_\infty\to A'_\infty\to 0.\]
		\end{proposition}
		\begin{proof}
			We work locally on opens $U'\subset E$ mapping isomorphically to $U$ in our cover $\mathfrak U$ of $A$. Write $\pi_n\colon E\to A_n$ for the quotient map. Since the rows in~\eqref{F-tower for E/M} are exact, and the transition maps on the left are isomorphisms, it follows that for each $n\in \mathbb{N}$, the quotient map $\pi_n$  sends the pullback $U'_n:=[p^n]^{-1}(U')$ isomorphically onto $U_n:=\pi_n(U'_n)\subseteq A_n$. Thus~\eqref{F-tower for E/M} is locally of the form
				\begin{equation}\label{F-tower for E/M-local}
				\begin{tikzcd}
				0 \arrow[r] & M_1\arrow[d, "\cong"] \arrow[r] &  M_1\times U_1' \arrow[d, "{[p]}"] \arrow[r] &U_1 \arrow[d, "{[p]_E}"] \arrow[r] & 0 \\
				0 \arrow[r] & M \arrow[r] & M\times U' \arrow[r] & U \arrow[r] & 0.
				\end{tikzcd}
				\end{equation}
			Let $U_\infty$ be the pullback of $U'$ along $E_\infty\to E$. We have $U_\infty\sim \varprojlim U'_n\cong \varprojlim U_n$. The system $(U_n)_{n\in \mathbb{N}}$ thus has a perfectoid tilde-limit. This shows that $\varprojlim A_n$ has a perfectoid tilde-limit. We can therefore apply Proposition \ref{SW Proposition 2.4.5} to get a morphism $E_\infty\rightarrow A'_\infty$, obtaining the desired short exact sequence in the limit over diagram \eqref{F-tower for E/M} since the transition maps in \eqref{F-tower for E/M-local} respect the splitting. 
		\end{proof}
	
We will keep the notation introduced in the above proof: $U'$ is an open of $E$ mapping isomorphically to $U\subset A$. The open $U'_n:=[p^n]^{-1}(U')\subset E$ maps isomorphically to its image $U_n\subset A_n$ and we have a commutative diagram with exact rows
\[
 		\begin{tikzcd}
		0 \arrow[r] & M_n \arrow[d, equal] \arrow[r] & M_n\times U'_n \arrow[d,hook] \arrow[r] &  U_n\arrow[d, hook] \arrow[r] & 0 \\
		0 \arrow[r] & M_n \arrow[r] & E \arrow[r,"{\pi_n}"] & A_n \arrow[r] & 0.
		\end{tikzcd}
\]
	
	To construct a tilde-limit for $\varprojlim A$, we use the fact that the horizontal maps in diagram~(\ref{p-multiplication tower of E/M splits into vertical and horizontal tower}) are all finite \'etale. They are even finite covering maps, in the following sense:
	\begin{lemma}\label{horizontal map is covering map}
		For any $n\geq 0$, the preimage of $U_n\subset A_n$ under the horizontal map $v^{n}\colon A\rightarrow A_n$ is isomorphic to $p^{rn}$ disjoint copies of $U_n$. More canonically, it is isomorphic to $D_{n}\times U_n$, where $D_n=M_n/M$ (see Remark~\ref{remark: Definition of the D_n}).
	\end{lemma}
	\begin{proof}
		For the first part, we observe that the map $v^n$ fits into a commutative diagram
			\begin{equation}\label{v-tower for E/M}
			\begin{tikzcd}
			0 \arrow[r] & M \arrow[d, hook] \arrow[r] & E \arrow[d,equal] \arrow[r] &  A\arrow[d, "{v^{n}}"] \arrow[r] & 0 \\
			0 \arrow[r] & M_n \arrow[r] & E \arrow[r] & A_n \arrow[r] & 0
			\end{tikzcd}
			\end{equation}
	where the map on the left is the natural inclusion. Upon restriction to $U_n\subset A_n$, this becomes
					\begin{equation}\label{v-tower for E/M-local}
		\begin{tikzcd}
		0 \arrow[r] & M \arrow[d, hook] \arrow[r] & M_n\times U'_n \arrow[d,equal] \arrow[r] &  (v^{n})^{-1}(U_n)\arrow[d, "{v^{n}}"] \arrow[r] & 0 \\
		0 \arrow[r] & M_n \arrow[r] & M_n\times U'_n \arrow[r] & U_n \arrow[r] & 0
		\end{tikzcd}
		\end{equation}
	and the claim follows the fact that $M$ is a discrete lattice of rank $r$, and from $U_n'\cong U_n$.
	\end{proof}
	\begin{definition}
The $[p]$-multiplication on $E$ maps $M_{n+1}$ onto $M_n$ and therefore the $[p]$-multiplication tower of $A$ induces a tower
 \[\cdots \xrightarrow{[p]}D_{n+1}=M_{n+1}/M\xrightarrow{[p]}D_n=M_n/M\rightarrow\cdots.\]
  Since $K$ is algebraically closed, the finite \'etale groups $D_n$ are already constant.  By Lemma~\ref{l:pro-finite-perfectoid-spaces}, there is a profinite perfectoid group $D_\infty$ such that
  \[D_\infty\sim \varprojlim_n D_n.\]
 \end{definition}
The quotient maps $M_n\to D_n=M_n\otimes_\Z \Z/p^n$ in the limit give rise to a closed immersion of perfectoid groups $M_\infty\hookrightarrow D_\infty= M_\infty\otimes_{\Z}\Z_p$.

Theorem~\ref{thm:main_thm_intro} is now part of the following theorem:	
	\begin{theorem}\label{tilde-limit of tilde-limits of partial towers is tilde-limit of whole tower}
		\begin{enumerate}
		\item There is a perfectoid space  $A_\infty$ which is a tilde-limit of $\varprojlim_{[p]}A$.
		\item It is independent up to canonical isomorphism of the auxiliary choice of lattices $M_n$ with $D_n=M_n/M$, but it remembers the choice as a pro-finite \'etale closed subgroup $D_\infty \subseteq A_\infty$. 
		\item The preimage of any $U\in \mathfrak U$ under the projection $A_\infty \rightarrow A$ is isomorphic to $D_\infty \times U_\infty$. 
		
		\item 	There is a natural map of analytic-locally split short exact sequences of perfectoid groups		
		\begin{center}
			\begin{tikzcd}
				0 \arrow[r] & M_{\infty} \arrow[r] \arrow[d, hook] & E_\infty \arrow[d, hook] \arrow[r] & A'_\infty \arrow[d,equal] \arrow[r] & 0 \\
				0 \arrow[r] & D_\infty \arrow[r] & A_\infty \arrow[r] & A'_\infty\arrow[r] & 0.
			\end{tikzcd}
		\end{center}
		\item One can describe $A_\infty$ as the associated fibre bundle
		\[A_\infty = D_\infty\times^{M_\infty}E_\infty.\]
		In particular, we have an analytic-locally split short exact sequence of perfectoid groups
		\[0\rightarrow M_\infty\rightarrow D_\infty \times E_\infty \rightarrow A_\infty\rightarrow 0\]
		where the map on the left is the antidiagonal embedding of $M_\infty$ into $D_\infty\times E_\infty$.
		\end{enumerate}
	\end{theorem}
	\begin{remark}
	We think of part (5) as the analogue of the Raynaud uniformisation
		\[0\to M\to E\to A\to 0.\]
	Here we note that while the map $E\to A$ is a quotient, in the limit over $[p]$ it becomes an immersion $E_\infty\hookrightarrow A_\infty$: The reason is that the projective system $(M,[p])$ has vanishing $\lim$ but non-vanishing $\mathrm{Rlim}^1$, for instance, when considered as abelian sheaves on perfectoid spaces for the pro-\'etale topology in the sense of \cite{etale_cohomology_of_diamonds} (assuming that $K$ is of characteristic $0$). A toy example of this phenomenon would be the inverse system over $[p]$ on the short exact sequence of groups 
	$0 \to\mathbb Z \to\mathbb R \to \mathbb R/\mathbb Z \to 0$
	whose limit yields an exact sequence
	\begin{center}
		\begin{tikzcd}
			0 \arrow[r] & 0 \arrow[r] & \mathbb R \arrow[r] & \varprojlim_{[p]}\mathbb R/\mathbb Z \arrow[r] & \varprojlim^1_{[p]}\mathbb Z = \mathbb Z_p/\mathbb Z\arrow[r] & 0.
		\end{tikzcd}
	\end{center}
	We therefore think the quotient $D_\infty/M_\infty=M_\infty\otimes_{\Z}(\Z_p/\Z$ implicit in part (5) as being an incarnation of $\mathrm{Rlim}^1_{[p]}M_\infty$.
\end{remark}
	\begin{proof}[Proof of Theorem~\ref{tilde-limit of tilde-limits of partial towers is tilde-limit of whole tower}]
We keep the notation from the proof of Proposition~\ref{M_infty->E_infty->A'_infty}: We have a cover of $A_n$ by open subsets $U_n$ and a  perfectoid open subspace $U_\infty\subseteq E_\infty$ for which $U_\infty\sim \varprojlim U_n$.
	 
By Lemma~\ref{horizontal map is covering map}, the restriction of diagram~(\ref{p-multiplication tower of E/M splits into vertical and horizontal tower}) to the open $U$ of the bottom $A$ becomes
\begin{equation*}
		\begin{tikzcd}[column sep={1.1cm,between origins},row sep={0.6cm,between origins}]
			\ddots  &  &  & \vdots &  & \vdots \\
			& D_2\times U_2 \arrow[rr,"v"] \arrow[rrdd, "{[p]}"'] &  & v^{-1}(U_2) \arrow[rr,"v"] \arrow[dd,"{[p]_E}"] &  & U_2 \arrow[dd,"{[p]_E}"] \\
			&  &  &  &  &  \\
			&  &  & D_1\times U_1 \arrow[rrdd, "{[p]}"'] \arrow[rr,"v"] &  & U_1 \arrow[dd,"{[p]_E}"] \\
			&  &  &  &  &  \\
			&  &  &  &  & U.
		\end{tikzcd}
\end{equation*}
Hence the restriction of the tower $\cdots\xrightarrow{[p]}A \xrightarrow{[p]}A\xrightarrow{[p]}A$ to $U$ becomes the inverse system 
	 \[\cdots\rightarrow D_{n+1}\times U_{n+1}\rightarrow D_{n}\times U_n\rightarrow \cdots.\]
	
	By Lemma~\ref{affinoid tilde-limits commute with fibre products} this inverse system has perfectoid tilde-limit $D_\infty \times U_\infty$. These local tilde-limits glue together to give the desired tilde-limit $A_\infty$. This proves parts (1), (2) and (3), and shows that the second row of part (4) is locally split and in particular exact.
	
	The first row in part (4) is from Proposition~\ref{M_infty->E_infty->A'_infty}. Part (5) follows immediately from part (4).
	\end{proof}
	\begin{remark}
		When working over a general perfectoid base field, the lattices $M_n$ may no longer be defined over $K$. Instead, one can show that the natural map $A[p^n]\times U_n\to V_n$ is an \'etale $E[p^n]$-torsor for the diagonal action where $V_n$ is the pullback of $U$ along $[p^n]\colon A\to A$. The point is that this torsor is split when $K$ is algebraically closed.
	\end{remark}

	%%%%%%%%%%%%%%%%%%%%%%%%%%%%%
	%%%%%%%%%%%%%%%%%%%%%%%%%%%%%
	%%%%%%%%%%%  APPLICATIONS
        %%%%%%%%%%%%%%%%%%%%%%%%%%%%%	
	%%%%%%%%%%%%%%%%%%%%%%%%%%%%%
	\section{Applications}
	In this section, we give three applications of our main result. For all of these, we assume that $K$ is of characteristic $0$, i.e.\ $K$ is an algebraically closed non-archimedean field extension of $\Q_p$.
	\subsection{Uniformisation}
	Our first application is a ``$p$-adic uniformisation'' of abelian varieties.
	Recall that any abelian variety $A$ over $\mathbb C$ of dimension $d$ has a uniformisation in terms of a complex torus $A\cong \mathbb C^d/\Lambda$ for some $2d$-dimensional lattice $\Lambda\subseteq \mathbb C^d$. More canonically, it admits the presentation
\[
 A\cong \operatorname{Lie} A/H_1(A,\Z).
\]
	
	We have the following analogue of this over $K$: Let $A$ be an abeloid variety over $K$ of dimension $d$, considered as a rigid space. Then in the limit over $n$, the short exact sequences
	\[ 0\to A[p^n]\to A\to A\to 0\]
	give rise to a short exact sequence of sheaves on perfectoid $K$-algebras with the pro-\'etale topology
	\[0\to T_pA \to A_\infty \to A\to 0.\]
	Using the language of diamonds from \cite{etale_cohomology_of_diamonds}, we then have:
	\begin{corollary}
		The diamond $A^{\diamond}$ associated to $A$ has a natural presentation
		\[A^{\diamond} = A_\infty/T_pA \]
		given by the perfectoid space $A_\infty$ with its pro-\'etale subgroup $T_pA$.
	\end{corollary}
	Here we think of $T_pA=H^{\et}_1(A,\Z_p)$ as the analogue of $H_1(A,\Z)$ in the complex setting.

	Of course this $p$-adic uniformisation of $A$ is very closely related to the uniformisation of the associated $p$-divisible group $A[p^\infty]$ described in \cite{SW} and \cite[\S4]{survey}: Indeed, in the language used there, we have a morphism of short exact sequences
	
	\begin{center}
		\begin{tikzcd}
			0 \arrow[r] & T_p(A[p^\infty]) \arrow[r] \arrow[d, equal] & \widetilde{A[p^\infty]} \arrow[d, hook] \arrow[r] & A[p^\infty] \arrow[d] \arrow[r] & 0 \\
			0 \arrow[r] & T_pA \arrow[r] & A_\infty \arrow[r] & A\arrow[r] & 0.
		\end{tikzcd}
	\end{center}
	\begin{remark}
	We note that for two abelian varieties $A$ and $B$ of dimension $d$, the universal covers $A_\infty$ and $B_\infty$ are different in general, so that this is only a ``uniformisation'' in a rather weak sense. However, they are canonically isomorphic if $A$ and $B$ are abelian varieties of good reduction with the same special fibre, or if $A$ and $B$ are $p$-power isogeneous, so that in these cases we can really think of $T_pA$ as a a $2d$-dimension $\mathbb Z_p$-lattice determining $A$.
	\end{remark}
	 There is a second closely related uniformisation which we discuss in \S5.3 below.
	\subsection{Stein property}
	As a second application, we can combine our main theorem with work of Reinecke to deduce the following Artin vanishing result:
	\begin{corollary}[Reinecke]
		Let $A$ be an abeloid variety over $K$. Let $L$ be a constructible sheaf of $\mathbb F_p$-modules on $A_{\et}$. Then for any $i>\dim A$,
		\[\textstyle\varinjlim_{n\in \mathbb N}H_{\et}^i(A,[p^n]^{\ast}L)=0.\]
	\end{corollary}
	\begin{proof}
	Due to Theorem~\ref{thm:main_thm_intro}, we can apply \cite[Theorem 3.3]{Reinecke} to the system $\dots \rightarrow A\xrightarrow{[p]}A$.
	\end{proof}
	A theorem of Artin and Grothendieck states if $X$ is an affine algebraic variety over $K$, then $H_{\et}^i(X,L)=0$ for any constructible $\mathbb F_p$-module $L$ and any $i>\dim A$. However, the rigid analogue of this statement is false in general. The point of the Corollary is that an analogue of this vanishing statement is true for the pullback of $L$ to $A_\infty$ in the following sense: Consider the morphism of sites $\nu\colon A_{\proet}\to A_{\et}$. Then by regarding $A_\infty$ as an object in $A_{\proet}$ via the pro-\'etale morphism $A_\infty\to A$, one can show
	\[H^i_{\proet}(A_\infty,\nu^{\ast}L)=\varinjlim_{n\in\N} H^i_{\et}(A,[p^n]^{\ast}L). \]
	By results from \cite{p-adic_Hodge}, the space $A_\infty$ has a ``Stein space''-like property in the sense that we have $H^j(V,\nu^{\ast}L)=0$ for any affinoid perfectoid $V\subseteq A_\infty$ and any $j>0$. One can use this to reduce to a computation in \v{C}ech cohomology, which shows that the left hand side vanishes for $i>\dim A$.
	\subsection{Universal perfectoid covers of curves}
	
As a third application, we describe how one can obtain universal perfectoid pro-\'etale covers of curves over $K$. This was first observed by Hansen \cite{Hansen-blog}.
	
	We start by recalling some background in a more general setting: Let $C$ be a connected smooth projective scheme over $K$, which we consider as an analytic adic space. By \cite[Theorem 3.1]{LutRiemann}, GAGA induces an equivalence of categories between finite \'etale covers of the scheme $C$ and finite \'etale covers of the adic space $C$. We can therefore fix a geometric base point $x:\Spa({K},\O_{{K}})\to C$ and study the usual \'etale fundamental group $\pi_1(C,x)$ using the language of adic spaces. 
	
	To prepare our discussion, we recall from \cite[\S3]{p-adic_Hodge} a few facts on the pro-finite-\'etale site of $C$: This is the category $C_{\profet}=\mathrm{pro-}C_{\fet}$ of small cofiltered inverse systems $(X_i)_{i\in I}$ in the finite \'etale site $C_{\fet}$. An object in $C_{\profet}$ is called perfectoid if there is a perfectoid tilde-limit $X_\infty\sim \varprojlim X_i$. Let $C_{\profet}^{\perf}$ be the full subcategory of perfectoid objects, and let $\mathrm{Perf}_C$ be the category of perfectoid spaces over $C$. Then the argument in the proof of \cite[Lemma 8.2.3]{berkeley} shows:
\begin{lemma}\label{l:profet-perf-tilde-limit-fully-faithful}
	Sending perfectoid pro-\'etale objects to their tilde-limits defines a fully faithful functor
	\[C_{\profet}^{\perf}\to  \mathrm{Perf}_{C},\quad (X_i)_{i\in I}\mapsto X_\infty\sim \varprojlim_{i\in I} X_i.\]
\end{lemma}
	We call the objects $X_\infty \to C$ in the essential image the pro-finite-\'etale perfectoid covers of $C$.
	\begin{proposition}[Proposition~3.5, \cite{p-adic_Hodge}]
		There is an equivalence of categories
		\[ F:C_{\mathrm{prof\acute{e}t}}\to \pi_{1}(C,x)\mathrm{-pfSets},\quad (Y_i)_{i\in I} \mapsto F(X):=\varprojlim_{i\in I}|Y_i\times_{C}x|=\varprojlim_{i\in I}\Hom_C(x,Y_i)\]
		from the pro-finite-\'etale site of $C$ to the category of profinite sets with continuous $\pi_{1}(C,x)$-action.
	\end{proposition}
	This restricts to the usual equivalence of finite \'etale covers to finite sets with continuous $\pi_{1}(C,x)$-action. 
	In particular, for every open subgroup $H\subseteq \pi_1(C,x)$, there is a corresponding finite \'etale morphism $C_H\to C$ from a connected scheme $C_H$, considered as an analytic adic space. For any two open subgroups $H_1\subseteq H_2\subseteq \pi_1(C,x)$, there is a natural map $C_{H_1}\to C_{H_2}$. For varying $H$, one therefore has a filtered inverse system $(C_H)_{H\subseteq \pi_1(C,x)}$ which we may regard as an object in $C_{\text{pro\'et}}$.
	\begin{corollary}[Hansen,\cite{Hansen-blog}]\label{c:universal-covers-of-curves}
		Let $C$ be a connected smooth projective curve of genus $g\geq 1$ over $K$, considered as an analytic adic space.
		\begin{enumerate}
		\item There is a perfectoid tilde-limit $\tilde{C} \sim \varprojlim_{H}C_H$ where $H$ ranges over the open subgroups of $ \pi_1(C,x)$. 
		\item The morphism $\tilde{C}\to C$ is a pro-\'etale $\pi_1(C,x)$-torsor. It is universal with this property in the sense that it represents the functor sending	 pro-finite-\'etale perfectoid covers $X\to C$ to the $\pi_1(C,x)$-module $F(X)=\mathrm{Hom}_C(x,X)$.
		\item For any $X\in C_{\mathrm{prof\acute{e}t}}$, for example for any finite \'etale $X\to C$, there is a natural isomorphism
		\[ X = \underline{F(X)}\times^{\pi_1(C,x)}\tilde{C}:=(\underline{F(X)}\times \tilde{C})/\pi_1(C,x).\]
		Here the right hand side is the categorical quotient in adic spaces for the antidiagonal action.
	\end{enumerate}
	\end{corollary}
\begin{remark}
Parts (2) and (3) say that we may reasonably regard $\tilde{C}\to C$ as the ``universal cover'' of $C$, in analogy with this notion in topology.
\end{remark}
The proof also works in the case that $C$ is an abelian variety. In this case, the \'etale fundamental group is simply the absolute Tate module  $\pi_1(A,x)=TA:=\varprojlim_{N\in \N} A[N](K)$. We then have:
\begin{corollary}\label{c:universal-covers-of-abelian-varieties}
	Let $A$ be an abelian variety over $K$, then there is a perfectoid tilde-limit \[\tilde{A}\sim \varprojlim_{[N]}A\]
	and the analogous statements of Corollary~\ref{c:universal-covers-of-curves}.2 and 3 hold for the $TA$-torsor $\tilde{A}\to A$. In particular, there is a natural isomorphism
	\[A^{\diamond}=\tilde{A}/TA =\tilde{A}/\pi_1(A,x).\]
\end{corollary}
\begin{proof}[Proof of Cor.~\ref{c:universal-covers-of-curves} and Cor.~\ref{c:universal-covers-of-abelian-varieties}]
To ease notation, let us abbreviate $G:=\pi_1(C,x)$.
	
We construct $\tilde{C}$ in two steps. 
The choice of the base point $x$ gives an embedding $\iota\colon C\rightarrow A$ of $C$ into its Jacobian. Let $C_n$ be the pullback of $C$ along the map $[p^n]\colon A\rightarrow A$. Combining our main theorem with \cite[Lemma II.2.2]{torsion}, we can pull back perfectoid tilde-limits along closed immersions and hence get a perfectoid space $C_\infty\sim \varprojlim C_n$ with a Zariski closed embedding $C_\infty\rightarrow A_\infty$.

We now use the fact that pro-\'etale covers of perfectoid spaces are again perfectoid to construct a perfectoid cover $\tilde{C}$ of $C_\infty$ that packages up the entire \'etale fundamental group of $C$. As we are assuming that $K$ has characteristic 0, the maps $[p^n]\colon A\rightarrow A$ are finite \'etale, so the induced covers $C_n\rightarrow C$ are finite \'etale. The inverse system 
\[\cdots \rightarrow C_n \rightarrow \cdots \rightarrow C_1\rightarrow C\] 
therefore corresponds to a chain of subgroups
\[\cdots < G_n <\cdots < G_1 < G=\pi_1(C,x).\]

For any open subgroup $H$ of $G$ corresponding to the finite \'etale cover $C_H\rightarrow C$, we have a decreasing sequence of positive integers 
\[\cdots \leq [G_n:G_n\cap H] \leq \cdots \leq [G_1:G_1\cap H]\leq [G:G\cap H]\]
which stabilises for $n\gg 0$.
So there is an integer $d$ such that for all $n\gg0 $, we have $[G_n:G_n\cap H]=d$. Translating back to the language of finite \'etale covers, we see that for such $n$, the map
\[C_{G_{n+1}\cap H}\rightarrow C_{G_n\cap H}\times_{C_{G_n}} C_{G_{n+1}}\]
coming from the universal property of the fibre product is an isomorphism: Both spaces are finite \'etale covers of $C_{G_{n+1}}$ of degree $d$, so the map is a finite \'etale cover of degree 1. This implies that the natural morphism $\varprojlim C_{G_n \cap H}\rightarrow \varprojlim C_{G_n}$ of objects of $C_{\profet}$ is finite \'etale in the sense of \cite[Definition 3.9]{p-adic_Hodge}. To simplify notation, we write this morphism as $C_{H,\infty}\rightarrow C_\infty$ (via Lemma~\ref{l:profet-perf-tilde-limit-fully-faithful}, one can also think of this as the corresponding map of perfectoid spaces).

We can now rewrite in $C_{\profet}$:
\[\displaystyle\varprojlim_{H\rightarrow 1} C_H=\varprojlim_{H\rightarrow 1}\displaystyle\varprojlim_{n\rightarrow \infty} C_{G_n\cap H}=\varprojlim_{H\rightarrow 1}C_{H,\infty}.\]
As the $C_{H,\infty}$ have compatible finite \'etale maps to $C_\infty$, we obtain a morphism in $C_{\profet}$
\[\varprojlim_{H\rightarrow 1}C_{H,\infty}\rightarrow C_\infty.\]

By \cite[Lemma 4.6]{p-adic_Hodge}, pro-finite-\'etale covers of perfectoid objects are again perfectoid, giving us the desired perfectoid space 
\[ \tilde{C}\sim\varprojlim_{H\rightarrow 1} C_H.\]
This completes the construction of $\tilde{C}$, and thus proves part (1).

To see part (2), we note that we can write $G=\varprojlim_N G/N$ where $N$ ranges through the normal open subgroups. These are precisely the subgroups for which $C_N\to C$ is already a finite \'etale $G/N$-torsor. Concretely, this means that the following natural morphism is already an isomorphism:
\[G/N\times_C C_N\to C_N\times_CC_N.\]
We note that we also have $\tilde{C} \sim \varprojlim_NC_N$, as normal open subgroups are cofinal in the inverse system of all open subgroups. In the limit, this shows that $\tilde{C}$ is a pro-finite-\'etale $G$-torsor.

To see that $F(X)=\Hom_C(\tilde{C},X)$, we recall that for any Galois cover $C_N\to C$ with a finite Galois map $C_N\to X$, we have $F(X)=\Hom_C(C_N,X)$. It therefore suffices to see that
\[\Hom_C(\tilde{C},X)=\varinjlim_{N}\Hom_C(C_N,X).\]
But this follows from Lemma~\ref{l:profet-perf-tilde-limit-fully-faithful}.

For (3), write $S=F(X)$, then it suffices to prove that the natural morphism \[\rho:\underline{S}\times\tilde{C}\to X\]
is a pro-finite-\'etale $G$-torsor for the antidiagonal action. Indeed, this implies that $X$ is the categorical quotient by the action of $G$: This is because the torsor property implies $\O_X = (\rho_{\ast}\O_{\underline{S}\times\tilde{C}})^{G}$ by combining \cite[{Lemma~2.24}]{CHJ} and \cite[{Theorem~8.2.3}]{KedlayaLiu-II}.

 Since connected components of $X$ correspond to $G$-orbits of $S$, we may reduce to the case where $X$ is connected and $G$ acts transitively on $\underline{S}$. By writing $X$ as a system of finite \'etale covers, we may further reduce to the case that $S$ is finite.  Fix $s\in S$ and let $H\subseteq G$ be the stabiliser of $s$, then $X=C_H$. It now suffices to show that for any normal open subgroup $N\subseteq G$ with $N\subseteq H$, the natural morphism
\[G/H\times C_N\to C_H\]
is a $G/N$-torsor, as the desired result will follow in the limit $N\to 1$. But this follows by Galois descent from the diagram
\begin{center}
\begin{tikzcd}
	G/H\times C_N \arrow[r]           & C_H           \\
	G/N \times C_N\arrow[r] \arrow[u] & C_N \arrow[u]
\end{tikzcd}
\end{center}
which is Cartesian as $C_N\to C_H$ is a finite \'etale $H/N$-torsor.
\end{proof}

	%%%%%%%%%%%%%%%%%%%%%%%%%%%%%
	%%%%%%%%%%%%%%%%%%%%%%%%%%%%%
	%%%%%%%%%%%  APPENDIX
        %%%%%%%%%%%%%%%%%%%%%%%%%%%%%	
	%%%%%%%%%%%%%%%%%%%%%%%%%%%%%
	
		\appendix
	\section{Fibre bundles of adic spaces}\label{s:appendix}
	In this appendix we review the theory of fibre bundles in the setting of analytic adic spaces. We will always implicitly assume that finite products of the adic spaces we work with exist: This is for instance the case when we work with rigid or perfectoid spaces, the cases we are most interested in.

	\begin{notation}
		In the following, if $\pi\colon E\rightarrow B$ is a morphism of adic spaces, then for an open subspace $U\subseteq B$ we denote $E|_U:=\pi^{-1}(U)\subseteq E$.
	\end{notation}
	\begin{definition}\label{definition principal T-bundle}
		Let $T$ be an adic group. Throughout we will assume that $T$ is commutative. 
		
		Let $F$ be an adic space with an action $m\colon T\times F\rightarrow F$.
		A morphism $\pi\colon E\rightarrow B$ of adic spaces is called a \textbf{fibre bundle with fibre $F$ and structure group $T$} if there is a cover $\mathfrak U$ of $B$ of open subspaces $U_i\subseteq B$ with isomorphisms $\varphi_i:F\times U_i \xrightarrow{\sim} E|_{U_i}$ which satisfy the following conditions:
		\begin{enumerate}[label=(\alph*)]
			\item For every $U_i\in \mathfrak U$, the following diagram commutes:
			\begin{center}
				\begin{tikzcd}
					F\times U_{i} \arrow[r, "\varphi_i"] \arrow[rd, "p_2"] & E|_{U_{i}} \arrow[d, "\pi"] & \phantom{T\times U_{ij}} \\
					& U_{i} & 
				\end{tikzcd}
			\end{center}
			\item For every two $U_i,U_j\in \mathfrak U$ with intersection $U_{ij}$, the commutative diagram
			\begin{center}
				\begin{tikzcd}
					F\times U_{ij} \arrow[r, "\varphi_i"] \arrow[rd, "p_2"] & E|_{U_{ij}} \arrow[d, "\pi"] & F\times U_{ij} \arrow[ld, "p_2"] \arrow[l, "\varphi_j"'] \\
					& U_{ij} & 
				\end{tikzcd}
			\end{center}
			produces an isomorphism $\phi_{ij}:=\varphi_j^{-1}\circ\varphi_i\colon F\times U_{ij}\rightarrow F\times U_{ij}$ with the following property: There exists a morphism $\psi_{ij}:U_{ij}\rightarrow T$ such that $\phi_{ij}$ coincides with the composite
			\[F\times U_{ij} \xrightarrow{\psi_{ij}\times \operatorname{id}\times\operatorname{id}} T\times F\times U_{ij}\xrightarrow{m\times \operatorname{id}} F\times U_{ij}.\]
		\end{enumerate}
	\end{definition}
	\begin{definition}
		When $F=T$ with the action on itself by left multiplication, then a fibre bundle $\pi\colon E\rightarrow B$ with fibre $T$ and structure group $T$ is called a \textbf{$T$-torsor}.
	\end{definition}
	
	\begin{example}
		The short exact sequence $0\rightarrow \overline{T}\rightarrow \overline{E}\xrightarrow{\pi} \overline{B}\rightarrow 0$ from \S\ref{Raynaud extensions as principal bundles of formal and rigid spaces} yields a $T$-torsor $\overline{E}\xrightarrow{\pi} \overline{B}$ by Lemma~\ref{formal Raynaud sequence is locally split}. Moreover, for any open subspace $U\subseteq \overline{B}$, the map $E|_U\rightarrow U$ is a $T$-torsor.
	\end{example}
	
	The $\phi_{ij}$ from condition (b) are determined by the maps $\psi_{ij}\colon U_{ij}\rightarrow T$. By glueing, one sees:
	\begin{lemma}\label{equivalent characterisation of principal $T$-bundle}
		Suppose we are given adic spaces $F$ and $B$ and an adic group $T$ with an action on $F$. Then fibre bundles $\pi\colon E\rightarrow B$ with fibre $F$ and structure group $T$ are equivalent to the data (up to refinement) of a cover $\mathfrak U$ of $B$ by open subspaces and morphisms $\psi_{ij}\colon U_{ij}\rightarrow T$ for all $U_i,U_j\in \mathfrak U$ that satisfy the cocycle condition $\psi_{jk}|_{U_{ijk}}\cdot \psi_{ij}|_{U_{ijk}}=\psi_{ik}|_{U_{ijk}}$ on the intersection $U_{ijk}:=U_i\cap U_j\cap U_k$.
	\end{lemma}
	\begin{lemma}
		Let $E\rightarrow B$ be a fibre bundle with fibre $F$ and structure group $T$. Then the natural $T$-action on $F\times U_{i}$ for each $i$ via the first factor glue to a natural $T$-action on $E$.
	\end{lemma}
	\begin{proof}
		This is immediate from condition (b).
	\end{proof}

	\begin{definition} \label{definition of Borel construction}
		Let $\pi\colon E\rightarrow B$ be a $T$-torsor. Let $F$ be an adic space with an action by $T$. Since the data in Lemma~\ref{equivalent characterisation of principal $T$-bundle} are completely independent of the fibre, the morphisms $\psi_{ij}\colon U_{ij}\rightarrow T$ by Lemma~\ref{equivalent characterisation of principal $T$-bundle} define a fibre bundle with fibre $F$ and structure group $T$ that we denote by $F\times^T E$. This is called the \textbf{associated bundle} or Borel--Weil construction.
	\end{definition}

\bibliographystyle{acm}
\bibliography{Arizona}

\begin{thebibliography}{10}

\bibitem{Bhatt}
{\sc Bhatt, B.}
\newblock The {H}odge--{T}ate decomposition via perfectoid spaces.
\newblock In {\em {P}erfectoid {S}paces: {L}ectures from the 2017 {A}rizona
  {W}inter {S}chool}, B.~Cais, Ed., vol.~242 of {\em Mathematical Surveys and
  Monographs}. American Mathematical Society, 2019.

\bibitem{Bosch-Lut}
{\sc Bosch, S., and L\"{u}tkebohmert, W.}
\newblock Stable reduction and uniformization of abelian varieties. {II}.
\newblock {\em Invent. Math. 78}, 2 (1984), 257--297.

\bibitem{BL}
{\sc Bosch, S., and L\"{u}tkebohmert, W.}
\newblock Degenerating abelian varieties.
\newblock {\em Topology 30}, 4 (1991), 653--698.

\bibitem{CHJ}
{\sc Chojecki, P., Hansen, D., and Johansson, C.}
\newblock {O}verconvergent modular forms and perfectoid {S}himura curves.
\newblock {\em Doc. Math. 22\/} (2017), 191--262.

\bibitem{Hansen_quotients}
{\sc Hansen, D.}
\newblock Quotients of adic spaces by finite groups.
\newblock {\em Math. Res. Letters\/}.
\newblock To appear.

\bibitem{Hansen-blog}
{\sc Hansen, D.}
\newblock Perfectoid universal covers of curves.
\newblock Blog entry,
  \url{https://arithmetica.wordpress.com/2015/09/27/perfectoid-universal-covers-of-curves/},
  2015.

\bibitem{huber2013etale}
{\sc Huber, R.}
\newblock {\em \'{E}tale cohomology of rigid analytic varieties and adic
  spaces}.
\newblock Aspects of Mathematics, E30. Friedr. Vieweg \& Sohn, Braunschweig,
  1996.

\bibitem{KedlayaLiu-II}
{\sc {Kedlaya}, K.~S., and {Liu}, R.}
\newblock {Relative {$p$}-adic Hodge theory, II: Imperfect period rings}.
\newblock {\em Preprint, arXiv:1602.06899\/} (2016).

\bibitem{LutRiemann}
{\sc L{\"{u}}tkebohmert, W.}
\newblock Riemann's existence problem for a $p$-adic field.
\newblock {\em Invent Math 111\/} (1993), 309--330.

\bibitem{Lut-survey}
{\sc L{\"{u}}tkebohmert, W.}
\newblock From {T}ate's elliptic curve to abeloid varieties.
\newblock {\em Pure Appl. Math. Q. 5}, 4, Special Issue: In honor of John Tate.
  Part 1 (2009), 1385--1427.

\bibitem{Lut}
{\sc L{\"{u}}tkebohmert, W.}
\newblock {\em Rigid geometry of curves and their {J}acobians}, vol.~61 of {\em
  Ergebnisse der Mathematik und ihrer Grenzgebiete. 3. Folge. A Series of
  Modern Surveys in Mathematics}.
\newblock Springer, Cham, 2016.

\bibitem{Pilloni-Stroh}
{\sc Pilloni, V., and Stroh, B.}
\newblock Cohomologie coh\'{e}rente et repr\'{e}sentations {G}aloisiennes.
\newblock {\em Ann. Math. Qu\'{e}. 40}, 1 (2016), 167--202.

\bibitem{Reinecke}
{\sc Reinecke, E.}
\newblock The cohomology of the moduli space of curves at infinite level.
\newblock {\em Preprint, arXiv:1911.07392\/} (2019).

\bibitem{perfectoid}
{\sc Scholze, P.}
\newblock Perfectoid spaces.
\newblock {\em Publ. Math. Inst. Hautes \'{E}tudes Sci. 116\/} (2012),
  245--313.

\bibitem{p-adic_Hodge}
{\sc Scholze, P.}
\newblock {$p$}-adic {H}odge theory for rigid-analytic varieties.
\newblock {\em Forum Math. Pi 1\/} (2013), e1, 77.

\bibitem{survey}
{\sc Scholze, P.}
\newblock Perfectoid spaces: a survey.
\newblock In {\em Current developments in mathematics 2012}. Int. Press,
  Somerville, MA, 2013, pp.~193--227.

\bibitem{scholzeICMproceedings}
{\sc Scholze, P.}
\newblock Perfectoid spaces and their applications.
\newblock {\em Proceedings of the ICM\/} (2014).

\bibitem{torsion}
{\sc Scholze, P.}
\newblock On torsion in the cohomology of locally symmetric varieties.
\newblock {\em Ann. of Math. (2) 182}, 3 (2015), 945--1066.

\bibitem{etale_cohomology_of_diamonds}
{\sc Scholze, P.}
\newblock \'{E}tale cohomology of diamonds.
\newblock {\em Preprint, arXiv:1709.0734\/} (2017).
\newblock Preprint.

\bibitem{SW}
{\sc Scholze, P., and Weinstein, J.}
\newblock Moduli of $p$-divisible groups.
\newblock {\em Camb. J. Math. 1}, 2 (2013), 145--237.

\bibitem{berkeley}
{\sc Scholze, P., and Weinstein, J.}
\newblock Berkeley lectures on $p$-adic geometry, revised version.
\newblock \url{http://www.math.uni-bonn.de/people/scholze/Berkeley.pdf}, 2017.

\end{thebibliography}

\end{document}